\documentclass[11pt]{amsart}
\usepackage{amssymb, amstext, amscd, amsmath, amssymb}
\usepackage{mathtools, paralist, color, dsfont, rotating}
\usepackage{verbatim}
\usepackage{enumerate,enumitem}
\usepackage{tikz-cd}
\usepackage{float}
\usepackage{setspace}
\usepackage{todonotes}

\usepackage{hyperref}
\usepackage{tikz}
\usepackage{mathdots}

\floatstyle{boxed} 
\restylefloat{figure}

\numberwithin{equation}{section}
\usepackage[a4paper]{geometry}
\usepackage{quoting}
\quotingsetup{vskip=.1in}
\quotingsetup{leftmargin=.17in}
\quotingsetup{rightmargin=.17in}

\let\OLDthebibliography\thebibliography
\renewcommand\thebibliography[1]{
  \OLDthebibliography{#1}
  \setlength{\parskip}{0pt}
  \setlength{\itemsep}{2pt plus 0.5ex}
}

\makeatletter
\def\@cite#1#2{{\m@th\upshape\bfseries%
[{#1\if@tempswa{\m@th\upshape\mdseries, #2}\fi}]}}
\makeatother
\theoremstyle{plain}
\newtheorem{theorem}{Theorem}[section]
\newtheorem{corollary}[theorem]{Corollary}
\newtheorem{proposition}[theorem]{Proposition}
\newtheorem{lemma}[theorem]{Lemma}
\theoremstyle{definition}
\newtheorem{definition}[theorem]{Definition}
\newtheorem{example}[theorem]{Example}

\newtheorem{remark}[theorem]{Remark}

\newtheorem{question}[theorem]{Question}

\theoremstyle{remark}

\mathtoolsset{centercolon}
\newcommand{\bbC}{{\mathbb{C}}}
\newcommand{\bbD}{{\mathbb{D}}}

\newcommand{\bbT}{{\mathbb{T}}}
\newcommand{\bbZ}{{\mathbb{Z}}}

\newcommand{\A}{{\mathcal{A}}}
\newcommand{\B}{{\mathcal{B}}}
\newcommand{\C}{{\mathcal{C}}}

\newcommand{\F}{{\mathcal{F}}}

\renewcommand{\H}{{\mathcal{H}}}

\newcommand{\K}{{\mathcal{K}}}
\renewcommand{\L}{{\mathcal{L}}}

\renewcommand{\O}{{\mathcal{O}}}

\renewcommand{\S}{{\mathcal{S}}}
\newcommand{\T}{{\mathcal{T}}}

\renewcommand{\phi}{\varphi}
\newcommand{\upchi}{{\raise.35ex\hbox{\ensuremath{\chi}}}}

\newcommand{\alg}{\operatorname{alg}}
\newcommand{\Aut}{\operatorname{Aut}}

\newcommand{\id}{{\operatorname{id}}}

\newcommand{\spn}{\operatorname{span}}

\newcommand{\ca}{\mathrm{C}^*}

\newcommand{\cmax}{\mathrm{C}^*_{\text{max}}}

\newcommand\0{\vec{0}}

\newcommand{\cstarlattice}{\text{C$^*$-Lat}}

\newcommand{\sDmax}{C^*_{\operatorname{sD-max}}}

\begin{document}
\title[Semi-Dirichlet operator algebras]{Crossed products and C*-covers of semi-Dirichlet operator algebras}

\author{Adam~Humeniuk}
\address{Department of Mathematics and Computing \\ Mount Royal University \\  Calgary, AB \\ Canada}
\email{ahumeniuk@mtroyal.ca}

\author[E.G. Katsoulis]{Elias~G.~Katsoulis}
\address {Department of Mathematics 
\\East Carolina University\\ Greenville, NC 27858\\USA}
\email{katsoulise@ecu.edu}

\author{Christopher~Ramsey}
\address {Department of Mathematics and Statistics
\\MacEwan University \\ Edmonton, AB \\Canada}
\email{ramseyc5@macewan.ca}

\begin{abstract}
In this paper, we show that the semi-Dirichlet C*-covers of a semi-Dirichlet operator algebra form a complete lattice, establishing that there is a maximal semi-Dirichlet C*-cover. Given an operator algebra dynamical system we prove a dilation theory that shows that the full crossed product is isomorphic to the relative full crossed product with respect to this maximal semi-Dirichlet cover. In this way, we can show that every semi-Dirichlet dynamical system has a semi-Dirichlet full crossed product.
\end{abstract}


\maketitle

\section{Introduction}

Dirichlet algebras were first introduced by Gleason in 1957 \cite{Gleason}. The name was given because ``[a] Dirichlet algebra is a function algebra for which the boundary fits smoothly into the [character space], in such a way that we can solve the analogue of the Dirichlet problem in harmonic functions" \cite[Section 3]{Gleason}. Interestingly, this is the same paper where he introduces the concept of parts (later called Gleason parts).
Arveson \cite{Arveson} generalized this notion of the closed self-adjoint operator space generated by an algebra being equal to a self-adjoint algebra to the non-commutative operator setting in his definition of subdiagonal algebras though this was in the context of von Neumann algebras. This was later generalized to say that an operator algebra $\A$ is Dirichlet if and only if $\overline{\A + \A^*}$ is a C$^*$-algebra in the appropriate context (everything will be defined precisely in the next section).

Davidson and the second author defined semi-Dirichlet algebras in \cite{DKDoc} to denote those operator algebras $\A$ where the self-adjoint operator space generated by $\A$ is well-behaved with respect to one of the two orders of multiplication, that is,
$\A^*\A \subseteq \overline{\A + \A^*}$, but perhaps not the other order. In particular, Dirichlet algebras are semi-Dirichlet. One of the main points of \cite{DKDoc} was that semi-Dirichlet operator algebras enjoy a rather well-behaved representation theory. This theme will be continued throughout this paper.

One particularly nice class of semi-Dirichlet algebras is the tensor algebras of C$^*$-correspondences. At the time of writing \cite{DKDoc} it was unknown whether this described all semi-Dirichlet algebras. The first Dirichlet algebra that is not completely isometrically isomorphic to a tensor algebra was described by Kakariadis \cite{Kakariadis}. The pursuit of finding non-tensor semi-Dirichlet examples was the initial driving force behind the joint work of the second and third authors \cite{KatRamMem, KatRamLimit}. All of this is carefully described in detail in Section 2 along with all the necessary background, definitions and theory pertaining to Dirichlet and semi-Dirichlet algebras and C$^*$-correspondences.
As well, possibly the simplest example of a non-tensor Dirichlet algebra is given (Proposition \ref{prop:HartzDirichlet}), and this operator algebra first appeared in the work of Hartz \cite{Hartz}:
\[
\left[\begin{matrix} A(\bbD) & 0 \\ C(\bbT) & A(\bbD)^* \end{matrix}\right],
\]
where $A(\bbD)$ is
the disc algebra and $C(\bbT)$ is the algebra of continuous functions
on the unit circle $\bbT$.

A standard method for understanding the structure of an operator algebra has been to consider the C$^*$-algebras that are generated by its completely isometric representations, namely its C$^*$-covers. This began with Arveson's conjecture in 1969 that there is always a smallest C$^*$-cover, called the C$^*$-envelope, which he was able to prove in many examples \cite{ArvEnvelope}. This conjecture was proven to be true by Hamana \cite{Hamana}. The complementary maximal C$^*$-cover was first investigated in depth by Blecher in 1999 \cite{Blech}. Under direct sums and boundary ideal quotients the collection of C$^*$-covers of an operator algebra, modulo C$^*$-cover isomorphism, was shown to be a complete lattice by Hamidi and Thompson \cite{Hamidi, Thompson}. The first and third authors defined equivalences of two operator algebras based on their complete lattices of C$^*$-covers to study how much information about the operator algebra is retained in its C$^*$-covers \cite{HumRam}. In particular, it was found that the complete lattice of C$^*$-covers cannot distinguish operator algebras up to completely isometric isomorphism. 

In Section 3, we study the C$^*$-covers of semi-Dirichlet operator algebras. A C$^*$-cover is called semi-Dirichlet if the operator algebra enjoys the semi-Dirichlet property in that cover. These covers turn out to be very nicely behaved as they arise from Shilov representations, which are compressions to an invariant subspace of a $*$-homomorphism of the C$^*$-envelope, Proposition \ref{prop:shilov}. Theorem \ref{thm:sDLattice} establishes that the subset of semi-Dirichlet C$^*$-covers forms a complete sublattice, implying that there is a maximal semi-Dirichlet C$^*$-cover. Most significantly, this maximal semi-Dirichlet C$^*$-cover corresponds to the Cuntz-Pimsner-Toeplitz algebra in the case of a tensor algebra of a C$^*$-correspondence, Theorem \ref{thm:sDmaxisToeplitz}.

With all of this theory, Section 4 studies the crossed products of semi-Dirichlet operator algebras. Crossed products of operator algebra dynamical systems were defined and studied by the second and third authors in \cite{KatRamMem}. Their study is deeply connected to the C$^*$-covers of the operator algebra and so the more information we have about these covers the more tractable the crossed products will be. Theorem \ref{thm:dilatestosD} shows that every covariant representation of a semi-Dirichlet operator algebra dynamical system dilates to a semi-Dirichlet covariant representation. This allows us to prove that the full crossed product is the same as the relative crossed product with the maximal semi-Dirichlet cover, Theorem \ref{thm:sDcrossedproduct}, with the consequence that the full crossed product is semi-Dirichlet. Lastly, Theorem \ref{thm:AbeliansDcrossed} shows that in the case when the group is abelian then the dynamical system is semi-Dirichlet if and only if the full crossed product is semi-Dirichlet, using Takai duality.

\section{Semi-Dirichlet operator algebras}\label{sec:semi-dirichlet}

As promised, we start with the careful definitions of Dirichlet and semi-Dirichlet. 

\begin{definition} Let $\A$ be an operator algebra with $\epsilon : \A \rightarrow C^*_e(\A)$ the completely isometric embedding of $\A$ into its C$^*$-envelope.
\begin{itemize}
    \item $\A$ is called {\bfseries semi-Dirichlet} if 
$\epsilon(\A)^*\epsilon(\A) \subseteq \overline{\epsilon(\A) + \epsilon(\A)^*}$.
\vskip 4 pt
\item $\A$ is called {\bfseries Dirichlet} if $\overline{\epsilon(\A) + \epsilon(\A)^*} = C^*_e(\A)$.
\end{itemize} 
\end{definition}

For an operator algebra $\A$ all completely isometric representations $\rho: \A \rightarrow B(H)$ yield completely isometrically isomorphic copies of the operator algebra $\rho(\A)^*$. So it is entirely justified to refer to $\A^*$ abstractly. In particular, every such $\rho$ uniquely determines a completely isometric representation $\tilde\rho : \A^* \rightarrow B(H)$ by $\tilde\rho(b) = \rho(b^*)^*$ for all $b\in \A^*$.

There is a complementary property to semi-Dirichlet using the other order, namely
\[
\epsilon(\A)\epsilon(\A)^* \subseteq \overline{\epsilon(\A) + \epsilon(\A)^*}.
\]
One could be tempted to call this the ``co-semi-Dirchlet property" (in the manner of isometry and co-isometry) or perhaps the ``semi-Dirichlet-$\ast$ property". To avoid these somewhat cumbersome names this is usually referred to as $\A^*$ is semi-Dirichlet since the above statement becomes
\[
\tilde\epsilon(\A^*)^*\tilde\epsilon(\A^*) \subseteq \overline{\tilde\epsilon(\A^*)^* + \tilde\epsilon(\A^*)}.
\]

\begin{lemma}\label{lem:sDisoinvariant}
Suppose $\A$ and $\B$ are completely isometrically isomorphic operator algebras. Then $\A$ is semi-Dirichlet if and only if $\B$ is semi-Dirichlet.
\end{lemma}
\begin{proof}
Suppose $\pi : \A \rightarrow \B$ is a completely isometric isomorphism and $\A$ is semi-Dirichlet. This induces a $*$-isomorphism $\tilde\pi: C^*_e(\A) \rightarrow C^*_e(\B)$ where
\[
\tilde\pi(\epsilon_\A(a)) = \epsilon_\B(\pi(a))
\]
for every $a\in \A$. Thus,
\begin{align*}
\epsilon_\B(\B)^*\epsilon_\B(\B) \ & = \ \tilde\pi\left(\epsilon_\A(\A)^*\epsilon_\A(\A)\right)
\\ & \subseteq \ \tilde\pi\left( \overline{\epsilon_\A(\A) + \epsilon_\A(\A)^*}\right)
\\ & = \overline{\epsilon_\B(\B) + \epsilon_\B(\B)^*}.\qedhere
\end{align*}
\end{proof}

\begin{definition}
A completely contractive representation $\varphi : \A \rightarrow B(H)$ is {\bfseries semi-Dirichlet} if 
\[
\varphi(\A)^*\varphi(\A) \subseteq \overline{\varphi(\A) + \varphi(\A)^*}.
\]
\end{definition}

\begin{proposition}[Proposition 4.2.i \& ii \cite{DKDoc}]\label{prop:DKDirichlet}
Let $\A$ be an operator algebra.
\begin{enumerate}
    \item $\A$ is Dirichlet if and only if $\A$ and $\A^*$ are semi-Dirichlet.
    \item $\A$ is semi-Dirichlet if and only if $\A$ has a completely isometric semi-Dirichlet representation.
\end{enumerate} 
\end{proposition}

\begin{proposition}[Lemma 5.2 \cite{KatRamMem}]\label{prop:KRDirichlet}
Let $\A$ be an operator algebra. If $\rho : \A \rightarrow B(H)$ is a completely isometric representation with $\overline{\rho(\A) + \rho(\A)^*}$ a C$^*$-algebra, then $\A$ is Dirichlet and there is a completely isometric isomorphism $\Phi: C^*(\rho(\A)) \rightarrow C^*_e(\A)$ such that $\Phi\rho = \epsilon$.
\end{proposition}

This result is particularly useful in identifying the C*-envelope, for example consider $C^*_e(A(\bbD)) = \overline{A(\bbD) + A(\bbD)^*} = C(\bbT)$.

As mentioned, Davidson and the second author introduced the definition of semi-Dirichlet in \cite{DKDoc} because it encapsulated the nice dilation properties of tensor algebras of C$^*$-correspondences. Much effort has been expended in showing that the class of semi-Dirichlet algebras is far broader than tensor algebras. To begin we need to remind the reader of some definitions.

A {\em C$^*$-correspondence} is a triple $(X, \C, \varphi_X)$ (but often just denoted $(X,\C)$) consisting of a C$^*$-algebra $\C$, a right Hilbert $\C$-module $(X, \langle\cdot,\cdot\rangle)$, and a non-degenerate $*$-homomorphism $\varphi_X : \C \rightarrow \L(X)$ into the C$^*$-algebra of adjointable operators on $X$. The point being that this is a generalization of a Hilbert space with
\begin{itemize}
\item $\langle x,y\rangle \in \C$
\item $\langle x,y\rangle = \langle y,x\rangle^*$
\item $\langle x,x\rangle \geq 0$ with equality if and only if $x=0$
\item $\langle x,yc\rangle = \langle x,y\rangle c$
\item $\langle x,\varphi_X(c)y\rangle = \langle \varphi_X(c^*)x,y\rangle$
\end{itemize}
for all $x,y\in X$ and $c\in \C$.

The tensor product C$^*$-correspondence $(X\otimes X, \C, \varphi_{X\otimes X})$ satisfies the following
\begin{itemize}
\item $xc\otimes y = x\otimes\varphi_X(c)y$
\item $(x\otimes y)c = x\otimes(yc)$
\item $\langle x_1\otimes x_2, y_1\otimes y_2\rangle = \langle y_1, \varphi_X(\langle x_1,x_2\rangle)y_2\rangle$
\item $\varphi_{X\otimes X}(c)(x\otimes y) = (\varphi_X(c) \otimes \id_X)(x\otimes y) = (\varphi_X(c)x)\otimes y$
\end{itemize}
for all $x,y\in X$ and $c\in \C$. From repeated uses of this one forms the Fock space \[
\F_X = \C \oplus X \oplus X^{\otimes 2} \oplus X^{\otimes 3} \oplus \cdots
\]
where $X^{\otimes n} = X^{\otimes n-1}\otimes X$, which gives $(\F_X,\C)$ as a C$^*$-correspondence. 

\begin{definition}
Consider the linear map $t_\infty : X \rightarrow \L(\F_X)$
\[
t_\infty(x)(c, x_{11}, x_{21}\otimes x_{22}, \dots)
\ = \ (xc, x\otimes x_{11}, x\otimes x_{21}\otimes x_{22}, \dots)
\]
and the $*$-homomorphism $\rho_\infty : \C \rightarrow \L(\F_X)$
\[
\rho_\infty(d)(c, x_{11}, x_{21}\otimes x_{22}, \dots)
\ = \ (dc, \varphi_X(d)x_{11}, \varphi_X(d)x_{21}\otimes x_{22}, \dots).
\]
Then $\T_{(X,\C)} = C^*(t_\infty(X), \rho_\infty(\C))$ is called the {\bfseries Cuntz-Pimsner-Toeplitz algebra} and \\ $\T_{(X,\C)}^+ = \overline{\alg}(t_\infty(X), \rho_\infty(\C))$ is called the {\bfseries tensor algebra}.
\end{definition}

It should be noted that these are often denoted $\T_X$ and $\T_X^+$. Both of these algebras enjoy universal properties. A {\bfseries completely contractive representation} $(t,\rho,H)$ of $(X,\C)$ consists of a completely contractive linear map $t:X\rightarrow B(H)$ and a $*$-homomorphism $\rho:\C \rightarrow B(H)$ such that
\[
t(\varphi_X(c)xd) = \rho(c)t(x)\rho(d),
\]
for all $x\in X$ and $c,d\in \C$. This representation is called {\bfseries isometric} if it additionally satisfies
\[
t(x)^*t(y) = \rho(\langle x,y\rangle)
\]
for all $x,y\in X$. Note that $(t_\infty, \rho_\infty)$ above is an isometric representation of $(X,\C)$.

\begin{proposition}[Muhly-Solel \cite{MS1}]\label{prop:integratedform}
If $(t,\rho,H)$ is a completely contractive representation of $(X,\C)$ then there exists a completely contractive homomorphism $\rho\rtimes t : \T_{(X,\C)}^+ \rightarrow \overline{\alg}(t(X),\rho(\C))$, called the {\bfseries integrated form} such that 
\[
(\rho\rtimes t)t_\infty = t \quad \textrm{and} \quad (\rho\rtimes t)\rho_\infty = \rho.
\]
If the representation is isometric then $\rho\rtimes t$ extends to a $*$-homomorphism of $\T_{(X,\C)}$ onto $C^*(t(X),\rho(\C))$.

Moreover, every completely contractive representation of $\T_{(X,\C)}^+$ is the integrated form of a completely contractive representation of $(X,\C)$.
\end{proposition}

All of this technology can be simplified in the following realization theorem.

\begin{theorem}[Katsoulis-Ramsey, Theorem 7.5 \cite{KatRamMem}]\label{thm:realization}
Let $\C \subset B(H)$ be a C$^*$-algebra and $X\subset B(H)$ a closed $\C$-bimodule such that $X^*X\subseteq \C$. Then $\overline{\alg}(X,\C)$ is completely isometrically isomorphic to $\T_{(X,\C)}^+$ if and only if the map
\[
\C \ni c \mapsto c\otimes I \ \textrm{and} \ X \ni x \mapsto x\otimes U,
\]
where $U$ is the bilateral  or unilateral shift, extends to a well-defined, completely contractive homomorphism of $\overline{\alg}(X,\C)$.
\end{theorem}

\begin{corollary}
Let $\C \subset B(H)$ be a C$^*$-algebra and $X\subset B(H)$ a closed $\C$-bimodule such that $X^*X\subseteq \C$. Then 
\[
\T_{(X,\C)}^+ \ \simeq \ \T_{(X\otimes U, \C\otimes I)}^+ \ \simeq \ \overline{\alg}(X\otimes U, \C\otimes I).
\]
\end{corollary}

The upshot of this is that if you have a concrete C$^*$-correspondence then it is very easy to find its tensor algebra, and then to show the tensor algebra is semi-Dirichlet. 

\begin{proposition}[Davidson-Katsoulis \cite{DKDoc}]\label{prop:tensorsD}
The tensor algebra of a C$^*$-correspondence is semi-Dirichlet.
\end{proposition}
\begin{proof}
Let $\C \subset B(H)$ be a C$^*$-algebra and $X\subset B(H)$ a closed $\C$-bimodule such that $X^*X\subseteq \C$. Then
\[
(X^*)^jX^k \ \subseteq \ \left\{ \begin{matrix} \C &  j=k \\ X^{k-j} & k > j \\ (X^*)^{j-k} & j > k\end{matrix} \right.
\]
By the previous corollary, it is easy now to obtain the desired result since
\[
\overline{\alg}(X\otimes U, \C\otimes I)^*\overline{\alg}(X\otimes U, \C\otimes I) \subseteq \overline{\overline{\alg}(X\otimes U, \C\otimes I) + \overline{\alg}(X\otimes U, \C\otimes I)^*}.\qedhere
\]
\end{proof}

A C$^*$-correspondence $(X,\C)$ is called a {\bfseries Hilbert bimodule} if it has a right $\C$-valued inner product $[\cdot,\cdot]$ that has the following compatibility with the original structure:
\[
\varphi_X([x,y])z = x\langle y,z\rangle
\]
for all $x,y,z\in X$.

\begin{theorem}[Kakariadis, Theorem 2.2 \cite{Kakariadis}]
The tensor algebra of a C$^*$-correspondence is Dirichlet if and only if the C$^*$-correspondence is a Hilbert bimodule.
\end{theorem}

This second inner product is akin to finding a concrete realization $X,\C\subseteq B(H)$ such that $XX^*$ and $X^*X\subseteq \C$.

\begin{example}\label{ex:sD}
Many nice operator algebras are tensor algebras of C$^*$-correspondences and so are semi-Dirichlet or Dirichlet:
\begin{enumerate}
\item The disc algebra $A(\bbD) \ \simeq \ \T_{(\bbC,\bbC)}^+$ is Dirichlet.
\item The noncommutative disc algebra $\A_d \ \simeq \ \T_{(\bbC^d, \bbC)}^+$ \cite{Popescu} is semi-Dirichlet, but not Dirichlet for $d\geq 2$.
\item The algebra of upper triangular matrices $T_n \ \simeq \ \T_{(\bbC^{n-1}, \bbC^n)}^+$, where $X$ is the superdiagonal and $\C$ is the diagonal of $T_n$, is Dirichlet.
\item $\left\{ \left[\begin{matrix} a&&b \\ &c&d\\ &&e\end{matrix} \right] : a,b,c,d,e\in\bbC \right\} \ \simeq \ \T_{(\bbC^2,\bbC^3)}^+$ where $\C$ is the diagonal and $X$ are the matrices with 0 diagonal, is semi-Dirichlet but not Dirichlet.
\item Every C$^*$-algebra $\C$ is a $\C$-bimodule over itself and thus $(\C,\C)$ is a Hilbert bimodule. Thus, 
\[
\T_{(\C,\C)}^+ \ \simeq \ \overline{\alg}(\C\otimes U, \C\otimes I)
\]
is Dirichlet, for $U$ the bilateral shift.
\item Suppose $(\C,\alpha)$ is a C$^*$-dynamical system, $\alpha:\C\rightarrow \C$ is a $*$-homomorphism. The semicrossed product algebra $\C \times_\alpha \bbZ^+_0$ is completely isometrically isomorphic to the tensor algebra of $(\C,\C, \varphi)$ where $\varphi(c)d = \alpha(c)d$.
\item The tensor algebra of a multivariable commutative C$^*$-dynamical system \cite[Section 5]{DKMulti}.
\item Tensor algebras arising from directed graphs. These are Dirichlet if and only if every vertex is the source and range of at most one edge \cite[Example 3.2]{Kakariadis}. 
\end{enumerate}
\end{example}

As mentioned before, when Davidson and the second author introduced the definition of semi-Dirichlet they did not know of any examples that weren't tensor algebras. The above list shows that there are a large variety of these algebras which can hide their tensor nature.

 The previous theorem was used to produce the first example of a Dirichlet (or semi-Dirichlet) algebra that is not a tensor algebra.

\begin{example}[Kakariadis, Proposition 3.10 \cite{Kakariadis}]
If $K = \{ z\in\bbC : |z-1| \leq 1 \ \textrm{or} \ |z+1| \leq 1\}$, then $P(K)$, the closure of the polynomials over $K$, is a Dirichlet algebra since \[
\overline{P(\partial K) + P(\partial K)^*} = C(\partial K).
\]
Moreover, $P(K)$ is not completely isometrically isomorphic to a tensor algebra of a C$^*$-correspondence. The proof uses the fact that there is only one way to construct a Hilbert bimodule of the form $(X,\bbC)$.
\end{example}

Direct limits of Dirichlet or semi-Dirichlet operator algebras are Dirichlet or semi-Dirichlet, respectively. 
Triangular limit algebras \cite{Power} then are a great source of Dirichlet or semi-Dirichlet operator algebras. These operator algebras are studied in the paper \cite{KatRamLimit} where isometric isomorphism to a tensor algebra of a C$^*$-correspondence is characterized by the fundamental relation being a tree structure with a well-behaved $\bbZ_0^+$-cocycle. 
The nicest type of triangular limit algebras are the so-called TUHF algebras (Triangular UHF or Triangular Uniformly Hyperfinite). Suppose $\C = \varinjlim (M_{k_n}, \varphi_n)$ is a UHF algebra, meaning $\varphi_n : M_{k_n} \rightarrow M_{k_{n+1}}$ is a unital, injective $*$-homomorphism, then if $\varphi(T_{k_n}) \subseteq T_{k_{n+1}}$ the algebra $\A = \varinjlim (T_{k_n}, \varphi_{k_n})$ is called TUHF.

\begin{example}[Katsoulis-Ramsey \cite{KatRamLimit}]
Recall, the standard embedding of $M_n$ into \\ $M_k(M_n)$ is
\[
M_n \ni A \ \mapsto \ I_k\otimes A = A\oplus A \oplus \cdots \oplus A \in M_k(M_n)
\]
whereas the refinement embedding of $M_n$ into $M_k(M_n)$ is tensoring in the other order
\[
M_n \ni A \ \mapsto \ A\otimes I_k \in M_k(M_n).
\]
For instance, the refinement embedding of $M_2$ into $M_4$ is
\[
\left[\begin{matrix} a&b\\ c&d\end{matrix}\right] \ \mapsto \ \left[\begin{matrix}a &&b \\ & a &&b \\ c&&d \\ &c&&d\end{matrix}\right].
\]
Any standard embedding TUHF algebra is completely isometrically isomorphic to a tensor algebra of a C$^*$-correspondence whereas any refinement embedding TUHF algebra is not isometrically isomorphic to a tensor algebra. In fact, one of the main heuristics of the paper \cite{KatRamLimit} is that if it isn't standard then it isn't tensor.
\end{example}

In \cite{KatRamMem} it was shown that the Dirichlet and semi-Dirichlet properties make the study of crossed product operator algebras much more tractable. This will be covered in detail later in the paper. For now it is good to point out that there are examples of crossed product operator algebras that are Dirichlet (or semi-Dirichlet and not Dirichlet) that are not tensor algebras \cite[Theorem 5.2, 5.4, and surrounding discussion]{KatRamMem}.

Lastly in this section, we turn to an interesting class of Dirichlet operator algebras.

\begin{proposition}
Suppose $\A$ is an operator algebra. Define the operator space
\[
\B \ = \ \left[\begin{matrix} \A & 0 \\ \overline{\A + \A^*} & \A^* \end{matrix}\right].
\]
Then $\A$ is semi-Dirichlet if and only if $\B$ is an operator algebra.
Moreover, the following are equivalent:
\begin{enumerate}
    \item $\A$ is Dirichlet
    \item $\B$ is a semi-Dirichlet operator algebra
    \item $\B^*$ is a semi-Dirichlet operator algebra
    \item $\B$ is a Dirichlet operator algebra
\end{enumerate}
\end{proposition}
\begin{proof}
The first result is due to the fact that $\A$ being semi-Dirichlet is equivalent to
\[
\A^*(\overline{\A + \A^*})\A \ \subseteq \ \overline{\A + \A^*}.
\]

For the second result, if $\A$ is Dirichlet, then by Proposition \ref{prop:DKDirichlet} $\overline{\B + \B^*} = M_2(C^*_e(\A))$, which by Proposition \ref{prop:KRDirichlet} implies that $\B$ is Dirichlet.

On the other hand, if $\B$ or $\B^*$ is semi-Dirichlet then $\B^*\B$ or $\B\B^*$ is in $\overline{\B + \B^*}$. Either one of these implies that $(\overline{\A + \A^*})^2 \subseteq \overline{\A + \A^*}$, which implies that $\A$  and $\A^*$ are semi-Dirichlet. Therefore, $\A$ is Dirichlet by Proposition \ref{prop:DKDirichlet}.
\end{proof}

Hartz used this algebra $\B$ in the case for $\A = A(\bbD)$ to give an example of a residually finite-dimensional (RFD) operator algebra where the maximal C$^*$-cover is not RFD \cite{Hartz}.

\begin{proposition}\label{prop:HartzDirichlet}
The Dirichlet operator algebra of Hartz
\[
\B\ = \ \left[\begin{matrix} A(\bbD) & 0 \\ C(\bbT) & A(\bbD)^* \end{matrix}\right]
\]
is not completely isometrically isomorphic to a tensor algebra of a C$^*$-correspondence. 
\end{proposition}
\begin{proof}
For the sake of contradiction suppose $\B$ is completely isometrically isomorphic to the tensor algebra $\T_{(X,\C)}^+$ of the C$^*$-correspondence $(X,\C)$. In particular, embed $X,\C \subset \B$ so that
\[
\B = \overline{\alg}(X, \C) \simeq \T_{(X,\C)}^+.
\] 
By Theorem \ref{thm:realization} the map
\[
\C \ni c \mapsto c\otimes I \quad \textrm{and} \quad X \ni x \mapsto x\otimes U,
\]
extends to a well-defined, completely contractive homomorphism on $\B$.

We know that $\C = \B \cap \B^* = \bbC E_{11} \oplus \bbC E_{22}$.
And so, for $i=1,2$, $(E_{ii}XE_{ii}, \bbC E_{ii})$ is a C$^*$-correspondence. Thus, from above this still implies 
\[
\bbC E_{ii} \ni c \mapsto c\otimes I \quad \textrm{and} \quad E_{ii}XE_{ii} \ni x \mapsto x\otimes U,
\]
extends to a well-defined, completely contractive homomorphism. Hence,
\[
\T^+_{(E_{11}XE_{11}, \bbC E_{11})} \simeq \overline{\alg}((E_{11}XE_{11}, \bbC E_{11}) = A(\bbD), \quad \textrm{and}
\]
\[
\T^+_{(E_{22}XE_{22}, \bbC E_{22})} \simeq \overline{\alg}((E_{22}XE_{22}, \bbC E_{22}) = A(\bbD)^*.
\]
This implies that there are unitaries $u,v\in A(\bbD)$, with spectrum $\bbT$, such that $E_{11}XE_{11} = \bbC u E_{11}$ and $E_{22}XE_{22} = \bbC v^* E_{22}$. In fact, it is shown in \cite[Corollary 2.7]{KatRamIso} that two semicrossed product algebras of C$^*$-dynamical systems are isometrically isomorphic if and only if the two systems are outer conjugate.

Now there must be something else in $X$, specifically in $E_{22}XE_{11}$, otherwise the tensor algebra is just $A(\bbD) \oplus A(\bbD)^*$. Let $0\neq f E_{21}\in  E_{22}XE_{11}$ for $f\in C(\bbT)$.

Now $u^{-1} = \bar u \in A(\bbD)^*$. Thus, 
\[
X \ni fE_{21} = \bar u E_{22} f E_{21} uE_{11} \in \T_{(X,\C)}^+X^2,
\]
and $X\cap \T_{(X,\C)}^+ X^2 = \{0\}$, a contradiction.
Therefore, $\B$ is not completely isometrically isomorphic to a tensor algebra of a C$^*$-correspondence.
\end{proof}

Most of the time this $\B$ algebra will not be a tensor algebra. One exception is when dealing with finite algebras.

\begin{example}
Consider the $\B$ algebra above for $\A = T_n$. Define
\[
J_n \ = \ \left[ \begin{matrix} &&1 \\ &\iddots \\ 1
\end{matrix} \right] \in M_n
\]
Then 
\[
\B = \left[\begin{matrix}
    J_n & 0 \\ 0 & I_n
\end{matrix}\right]\left[\begin{matrix}
    T_n & 0 \\ M_n & T_n^*
\end{matrix}\right]\left[\begin{matrix}
    J_n & 0 \\ 0 & I_n
\end{matrix}\right] = \left[\begin{matrix}
    T_n^* & 0 \\ M_n & T_n^*
\end{matrix}\right] = T_{2n}^* = J_{2n}T_{2n}J_{2n}
\]
is a tensor algebra of a C$^*$-correspondence.
\end{example}

\section{Semi-Dirichlet C*-covers}

We move now to considering representation theory. However, the collection of all semi-Dirichlet representations may be ill-behaved. Instead we will focus on the completely isometric semi-Dirichlet representations.

\begin{definition}
Suppose $\A$ is an operator algebra. A {\bfseries C$^*$-cover} of $\A$ is $(\C,\iota)$ where $\iota:\A \rightarrow \C$ is a completely isometric homomorphism and $\C = C^*(\iota(\A))$.
\end{definition}

Suppose $(\C_j,\iota_j), j=1,2$ are C$^*$-covers of an operator algebra $\A$. A {\bfseries morphism of C$^*$-covers} is a $*$-homomorphism $\pi : \C_2\rightarrow \C_1$ such that $\pi\iota_1 = \iota_2$. If this happens, $\pi$ is unique and it is denoted $(\C_1,\iota_1) \preceq (\C_2,\iota_2)$. Thus, if $(\C_1,\iota_1) \preceq (\C_2,\iota_2) \preceq (\C_1,\iota_1)$ then $\pi$ is a $*$-isomorphism. This is denoted $(\C_1,\iota_1) \sim (\C_2,\iota_2)$ and is an equivalence relation.
For every C$^*$-cover $(\C,\iota)$ we have
\[
(C^*_e(\A), \epsilon) \ \preceq \ (\C,\iota) \ \preceq \ (\cmax(\A),\mu).
\]

The collection of equivalence classes $[\C,\iota]$ under $\sim$ is denoted $\cstarlattice(\A)$. This is a set since it is in one-to-one correspondence with the boundary ideals for $\A$ in $\cmax(\A)$. The induced partial order under $\preceq$ is in fact a complete lattice order. The join of an arbitrary collection $[\C_\lambda, \iota_\lambda] \in \cstarlattice(\A), \lambda\in \Lambda$ is given by the direct sum of the representations
\[
\bigvee_{\lambda} [\C_\lambda, \iota_\lambda] = \left[ C^*\left(\left(\bigoplus_\lambda \iota_\lambda\right)(\A)\right),  \bigoplus_\lambda \iota_\lambda\right]
\]
while the meet is obtained by taking the quotient by the ideal of the join that is generated by all of the individual ideals:
\[
\bigwedge_\lambda [\C_\lambda, \iota_\lambda] = \left. \left(\bigvee_{\lambda} [\C_\lambda, \iota_\lambda]\right) \middle/ \ \overline{\sum_\lambda \ker\left( \bigvee_{\mu} [\C_\mu, \iota_\mu] \rightarrow [\C_\lambda, \iota_\lambda]\right)}\right.
\]
This theory was developed in \cite[Chapter 2]{BlechLM}, \cite{Hamidi}, and \cite{Thompson}. 

\begin{remark}
The first and third authors write extensively about the lattice of C$^*$-covers and its equivalences in \cite{HumRam}. We note here that the semi-Dirichlet property is not preserved by any of these equivalences. In particular, \cite[Proposition 2.8]{HumRam} gives that $\A$ and $\A^*$ are lattice intertwined, the strongest of the equivalences, but $\A$ can be semi-Dirichlet while $\A^*$ is not, e.g. (iv) in Example \ref{ex:sD}.
\end{remark}

\begin{definition}
A C$^*$-cover $(\C, \iota)$ will be called {\bfseries semi-Dirichlet} if $\iota$  is a semi-Dirichlet completely isometric representation of $\A$.
\end{definition}

 By Proposition \ref{prop:DKDirichlet} an operator algebra is semi-Dirichlet if and only if it has a semi-Dirichlet C$^*$-cover. 
 

\begin{definition}
A representation $\rho : \A \rightarrow B(H)$ of an operator algebra $\A$ is called {\bfseries Shilov} if there is a $*$-homomorphism $\pi : C^*_e(\A) \rightarrow B(K)$ where $H\subseteq K$ is an invariant subspace of $\pi(A)$ such that $\rho(a) = \pi(a)|_H$ for all $a\in\A$.
\end{definition}

Davidson and the second author used Shilov representations in \cite{DKDoc} but this idea is older, cf. \cite[Definition 4.1]{MS1}.

\begin{proposition}\label{prop:shilov}
Suppose $\rho$ is a completely isometric representation of a semi-Dirichlet operator algebra. Then $\rho$ is a semi-Dirichlet representation if and only if it is a Shilov representation.
\end{proposition}
\begin{proof}
It is known that every Shilov representation of a semi-Dirichlet operator algebra is semi-Dirichlet \cite[Lemma 4.2]{DKDoc}.

Conversely, suppose $\rho:\A \rightarrow B(H)$ is a completely isometric, semi-Dirichlet representation. If the representation is not cyclic, meaning $\overline{\rho(\A)H} = H$, then compress to the reducing subspace $\overline{\rho(\A)H}$. The compression will still be completely isometric and semi-Dirichlet. So without loss of generality, assume that $\rho$ is cyclic.

By \cite[Theorem 4.6]{DKDoc} and its proof $\rho$ has a unique minimal extremal coextension $\sigma:\A \rightarrow B(K)$
\[
\sigma = \left[\begin{matrix} \rho & 0 \\ *&*\end{matrix}\right]
\]
that is Shilov and cyclic, $K = \overline{\sigma(\A)H}$. Continuing with their argument, let $\pi$ be a maximal dilation of $\sigma$ on $B(L)$ with $K\subseteq L$ invariant and $\pi(a)|_K = \sigma(a)$ for all $a\in\A$. This maximal dilation extends to a $*$-homomorphism of $C^*_e(\A)$, actually a $*$-isomorphism since $\rho$ is completely isometric.

Now, suppose $a_1,a_2\in \A$. Because $\rho$ is semi-Dirichlet there exist $b_n,c_n\in \A$ such that 
\[
\rho(a_2)^*\rho(a_1) = \lim_{n\rightarrow \infty} \rho(b_n) + \rho(c_n)^*
\]
Suppose $q: C^*(\rho(\A)) \rightarrow C^*_e(\A)$ is the canonical quotient $*$-homomorphism such that $q(\rho(a)) = \epsilon(a)$ for all $a\in \A$, or in other words, the unique morphism of C$^*$-covers \\ $(C^*(\rho(\A)), \rho)$ to $(C^*_e(\A), \epsilon)$. This implies that by applying $q$ to the equation above we get that
\[
\epsilon(a_2)^*\epsilon(a_1) = \lim_{n\rightarrow \infty} \epsilon(b_n) + \epsilon(c_n)^*.
\]
Hence, by similar calculations found in \cite[Lemma 4.8]{DKDoc}, for $h_1,h_2\in H$ and omitting the use of $\epsilon$,
\begin{align*}
    \langle \sigma(a_1)h_1, \sigma(a_2)h_2\rangle 
    \ & = \ \langle \pi(a_1)h_1, \pi(a_2)h_2\rangle
\\ & = \ \langle \pi(a_2^*a_1)h_1, h_2\rangle
\\ & = \ \lim_{n\rightarrow \infty} \langle (\pi(b_n) + \pi(c_n)^*)h_1,h_2\rangle
\\ & = \ \lim_{n\rightarrow \infty} \langle (\rho(b_n) + \rho(c_n)^*)h_1, h_2\rangle
\\ & = \ \langle \rho(a_2)^*\rho(a_1)h_1, h_2\rangle
\\ & = \ \langle \rho(a_1)h_1, \rho(a_2)h_2\rangle,
\end{align*}
where the middle step happens since the compression of $\pi(\A + \A^*)$ to $H$ is $\rho(\A) + \rho(\A)^*$.
This implies that 
\[
\sigma = \left[\begin{matrix} \rho & 0 \\ 0 & * \end{matrix}\right].
\]
Therefore, $\sigma = \rho$ since $\sigma$ is minimal. Therefore, $\rho$ is Shilov.
\end{proof}




We say a subset $S$ of a lattice $(L,\preceq)$ is {\bfseries downward-closed} if whenever $a\in L$, $b\in S$ and $a\preceq b$ then $a\in S$.

\begin{theorem}\label{thm:sDLattice}
The semi-Dirichlet C$^*$-covers of an operator algebra $\A$ form a, possibly empty, downward-closed complete sublattice of $\cstarlattice(\A)$. 
\end{theorem}
\begin{proof}
Assume that $\A$ is semi-Dirichlet, otherwise there are no semi-Dirichlet C$^*$-covers. 

Suppose $(\C_1, \iota_1) \preceq (\C_2, \iota_2)$ are in $\cstarlattice(\A)$ with $(\C_2, \iota_2)$ a semi-Dirichlet C$^*$-cover. Let $\pi$ be the unique morphism of C$^*$-covers. Then 
\begin{align*}
\iota_1(\A)^*\iota_1(\A) \ & = \ \pi\iota_2(\A)^*\pi\iota_2(\A)
\\ & = \ \pi(\iota_2(\A)^*\iota_2(\A))
\\ & \subseteq \pi(\overline{\iota_2(\A) + \iota_2(A)^*})
\\ & = \overline{\iota_1(\A) + \iota_1(\A)^*}
\end{align*}
and so $(\C_1, \iota_1)$ is also a semi-Dirichlet C$^*$-cover. This implies closure under arbitrary meets and downward-closedness.

For arbitrary joins, suppose $(\C_\lambda, \iota_\lambda)$ is a semi-Dirichlet C$^*$-cover of $\A$ with $\iota: \A \rightarrow B(H_\lambda)$ for every $\lambda\in\Lambda$. By Proposition \ref{prop:shilov} these are all Shilov representations, meaning there exist Hilbert spaces $K_\lambda$ and $*$-homomorphisms $\pi_\lambda : C^*_e(\A) \rightarrow B(K_\lambda)$ such that $H_\lambda \subseteq K_\lambda$ is an invariant subspace of $\pi_\lambda(\A)$ with
\[
\iota_\lambda(a) = \pi_\lambda(a)|_{H_\lambda} \quad\text{for all } a\in\A.
\]
Define the Hilbert space $K = \bigoplus_{\lambda\in\Lambda} K_\lambda$ and
the $*$-homomorphism
\[
\pi = \bigoplus_{\lambda\in\Lambda} \pi_\lambda: C^*_e(\A) \rightarrow B(K). 
\]
Notice that every $H_\lambda$ is an invariant subspace of $\pi(\A)$, because $K_\lambda$ is a reducing subspace. Thus, 
$H = \bigoplus_{\lambda\in\Lambda} H_\lambda \subseteq K$ is an invariant subspace of $\pi(\A)$ with
\[
\pi(a)|_{H} = \bigoplus_{\lambda\in\Lambda} \pi_\lambda(a)|_{H_\lambda}
= \bigoplus_{\lambda\in\Lambda} \iota_\lambda(a) = \left(\bigvee_{\lambda\in\Lambda} \iota_\lambda\right)(a).
\]
Therefore, the join is a Shilov representation that is completely isometric and so by Proposition \ref{prop:shilov} the join is semi-Dirichlet.
\end{proof}

\begin{definition} From the previous theorem, a semi-Dirichlet operator algebra $\A$ has a maximal semi-Dirichlet C$^*$-cover in $\cstarlattice(\A)$. This will be denoted $(\sDmax(\A), \nu)$.
\end{definition}

The minimal element of this complete sublattice is of course the C$^*$-envelope $(C^*_e(\A), \epsilon)$. Do note that we already have the previous result for Dirichlet operator algebras by Proposition \ref{prop:KRDirichlet}, which says that the complete lattice of Dirichlet C$^*$-covers is just the one point set $\{[C^*_e(\A), \epsilon]\}$.

The sublattice of semi-Dirichlet C$^*$-covers is preserved by completely isometric isomorphism and so as a companion to Lemma \ref{lem:sDisoinvariant} we have the following:

\begin{corollary}
Suppose $\A$ and $\B$ are semi-Dirichlet operator algebras. If $\pi : \A \rightarrow \B$ is a completely isometric isomorphism then there is a $*$-isomorphism $\tilde \pi : \sDmax(\A) \rightarrow \sDmax(\B)$ such that
\[
\tilde\pi(\nu_\A(a)) = \nu_\B(\pi(a))
\]
for all $a\in \A$.
\end{corollary}

While Shilov implies semi-Dirichlet, the converse is not true as seen in the next example. 

\begin{example}
It is shown in 
\cite{Loring} (cf. \cite[Example 2.4]{Blech}) that
\[
\cmax(T_2) = \left\{ f \in M_2(C([0,1])) : f(0) \ \textrm{is diagonal} \right\}
\] where the completely isometric embedding is
\[
\mu\left(\left[\begin{matrix}a & b \\ 0 & c\end{matrix}\right]\right) \ = \ \left[\begin{matrix}a 1 & b\sqrt{x} \\ 0 & c 1\end{matrix}\right].
\]
Note that $(\cmax(T_2),\mu)$ is certainly not a semi-Dirichlet C$^*$-cover:
\[
\left[\begin{matrix}0 & \sqrt{x} \\ 0 &  0\end{matrix}\right]^*\left[\begin{matrix}0 & \sqrt{x} \\ 0 & 0\end{matrix}\right] \ = \ \left[\begin{matrix}0 & 0 \\ 0 & x\end{matrix}\right] \ \notin \ \mu(T_2) + \mu(T_2)^*,
\]
where the failure occurs in the (2,2)-entry.
Intriguingly, $\mu$ is a direct sum of semi-Dirichlet representations:
\[
\left[\begin{matrix}a & b \\ 0 & c\end{matrix}\right] \ \mapsto \ 
\bigoplus_{\lambda\in [0,1]} \left[\begin{matrix}a & b\sqrt{\lambda} \\ 0 & c\end{matrix}\right].
\]
This shows the necessity of staying within the completely isometric semi-Dirichlet representations. 

It will be established by the following proposition that $(\sDmax(T_2), \nu)$ is given by $\sDmax(T_2) = \bbC \oplus M_2$ and 
\[
\nu \left(\left[ \begin{matrix} a & b \\ 0 & c\end{matrix}
\right]\right) \ = \ \left[
\begin{matrix} a \\ & a & b \\ & 0 & c \end{matrix}\right].
\]
\end{example}


In the case of the tensor algebra of a C$^*$-correspondence the maximal semi-Dirichlet C$^*$-cover is well-understood.

\begin{theorem}\label{thm:sDmaxisToeplitz}
If $(X,\C)$ is a C$^*$-correspondence then $\sDmax(\T_{(X,\C)}^+) \simeq \T_X$ as C$^*$-covers of $\T^+_{(X,\C)}$.
\end{theorem}
\begin{proof}
In the same way as the proof of Proposition \ref{prop:tensorsD} we see that because $t_\infty(x)^*t_\infty(y) = \rho_\infty(\langle x,y\rangle)$ for all $x,y\in X$ then $\T_{(X,\C)}^+$ has the semi-Dirichlet property as a subalgebra of $\T_{(X,\C)}$. In fact, this is true for every isometric representation of $(X,\C)$.

Conversely, if $\pi$ is a completely isometric semi-Dirichlet representation of $\T_{(X,\C)}^+$, then since it is completely contractive it is the integrated form of a completely contractive representation of $(X,\C)$, that is $\pi = \rho\rtimes t$ by Proposition \ref{prop:integratedform}. By Muhly and Solel \cite{MS1} this coextends to an isometric representation $(\rho', t', H')$ where $\rho'\rtimes t'$ coextends $\rho\rtimes t$. However, by Proposition \ref{prop:shilov} $\pi = \rho\rtimes t$ is Shilov which is the same thing as an extremal coextension for a semi-Dirichlet algebra. Hence, $\rho'\rtimes t' = (\rho\rtimes t) \oplus \pi'$ which implies that 
\[
t(x)^*t(y) = P_H t'(x)^*t'(y)P_H = P_H \sigma'(\langle x,y\rangle)P_H 
= \sigma(\langle x,y\rangle).
\]
Thus, $\pi$ arises from an isometric representation of $(X,\C)$. Therefore, by the universal property of $\T_{(X,\C)}$ we have $C^*(\pi(\T_{(X,\C)}^+)) \preceq \T_{(X,\C)}$ as C$^*$-covers.
\end{proof}

This result can be used to show that the lattice of semi-Dirichlet covers being one point does not imply the operator algebra is Dirichlet.

\begin{example}
Consider the graph with one vertex and a countable infinite number of loops. If the graph correspondence is $(X,\C)$ then $\T_{(X,\C)}^+$ is sometimes denoted $\A_\infty$, the infinite version of the non-commutative disc algebra. By Example \ref{ex:sD} part (viii) this operator algebra is semi-Dirichlet but not Dirichlet. However, it is known that 
\[
\sDmax(\T_{(X,\C)}^+) \simeq \T_{(X,\C)} \simeq \O_\infty \simeq C^*_e(\T_{(X,\C)}^+),
\]
via isomorphisms of C$^*$-covers.
\end{example}

This does raise the follow question:

\begin{question}
When is the lattice of semi-Dirichlet C*-covers a singleton? As mentioned, Dirichlet algebras and the previous example fall into this situation but is a general description such operator algebras?
\end{question}

Complementary to this is the opposite situation:

\begin{question}
When is the lattice of semi-Dirichlet C*-covers the whole lattice of C*-covers? Or equivalently, when is the maximal semi-Dirichlet C*-cover equal to the maximal C*-cover?
\end{question}

Apart from the trivial case of when $\A$ is a C*-algebra the authors are unaware of any examples satisfying the previous question.



\section{Coextending semi-Dirichlet dynamics}

 A {\bfseries dynamical system} $(\A, G, \alpha)$ is made up of an approximately unital operator algebra $\A$, a locally compact Hausdorff group $G$, and a representation $\alpha : G \rightarrow \Aut(\A)$ into the completely isometric automorphisms of $\A$ that is continuous in the point-norm topology. Note that if $\A$ is a C$^*$-algebra then $\alpha$ maps into the $*$-automorphisms of $\A$. The theory of these systems and their related crossed product operator algebras is thoroughly studied in \cite{KatRamMem} and the interested reader should consult that monograph for further reading. 

 For a given dynamical system $(\A, G, \alpha)$ a C$^*$-cover $(\C, \iota)$ is called {\bfseries $\alpha$-admissible} if there is a C$^*$-dynamical system $(\C, G, \beta)$ such that
 \[
 \beta_g(\iota(a)) = \iota(\alpha_g(a))
 \]
 for all $g\in G, a\in\A$. This $\beta$ is uniquely defined \cite[Lemma 2.2]{KatRamHN} and so is often referred to as $\alpha$. The $\alpha$-admissible C$^*$-covers for $(\A, G,\alpha)$ form a complete sublattice of $\cstarlattice(\A)$ containing $C^*_e(\A)$ and $\cmax(\A)$ \cite{Hamidi} \cite[Lemma 3.4]{KatRamMem}. However, not every C$^*$-cover is $\alpha$-admissible \cite{KatRamHN, Hamidi}.

\begin{proposition}\label{prop:sDmaxAdmissible}
If $(\A, G,\alpha)$ is a dynamical system, then $(\sDmax(\A),\nu)$ is an $\alpha$-admissible C$^*$-cover.
\end{proposition}
\begin{proof}
The proof largely follows the argument of \cite[Lemma 3.4]{KatRamMem}. 
First, let $\gamma \in \Aut(\A)$. Then $(C^*(\nu\gamma(\A)), \nu\gamma)$ is a C$^*$-cover. It is easy to see that $\sDmax(\A) = C^*(\nu(\A)) = C^*(\nu\gamma(\A))$ but we need to prove that the two covers are equivalent.
Now 
\begin{align*}
\nu\gamma(\A)^*\nu\gamma(\A)  & \ = \ \nu(\A)^*\nu(\A)
\\ & \ \subseteq \ \overline{\nu(\A) + \nu(\A)^*}
\\ & \ = \ \overline{\nu\gamma(\A) + \nu\gamma(\A)^*}
\end{align*}
since $\nu$ is a semi-Dirichlet representation.
Thus, by the maximal (or universal) property
\[
(C^*(\nu\gamma(\A)), \nu\gamma) \ \preceq \ (\sDmax(\A),\nu)
\]
and so there exists a morphism of C$^*$-covers, that is a surjective $*$-homomorphism $\varphi : \sDmax(\A) \rightarrow C^*(\nu\gamma(\A))$ such that 
$\varphi\nu = \nu\gamma$. Do the same thing for $\gamma^{-1}$, finding $\varphi': \sDmax(\A) \rightarrow C^*(\nu\gamma^{-1}(\A))$ such that 
$\varphi'\nu = \nu\gamma^{-1}$.
Thus, 
\[
\varphi\varphi'\nu = \varphi\nu\gamma^{-1} = \nu\gamma\gamma^{-1} = \nu
\]
and $\varphi'\varphi\nu = \nu$ similarly.
This  implies that $\varphi\varphi'|_{\nu(\A)} = \id|_{\nu(\A)} = \varphi'\varphi|_{\nu(\A)}$ and so $\varphi\varphi' = \varphi'\varphi = \id$ since $\nu(\A)$ generates $\sDmax(\A)$. Therefore, $\varphi \in \Aut(\sDmax(\A))$ with $\varphi\nu = \nu\gamma$.

Lastly, from above and following in the same way as the proof of \cite[Lemma 3.4]{KatRamMem} the group representation
\[
G \ni g \mapsto \nu\alpha_g\nu^{-1}|_{\nu(\A)} \in \Aut(\nu(\A))
\]
extends uniquely to a group representation $\alpha: G\rightarrow \sDmax(\A)$.
\end{proof}

A {\bfseries covariant representation} $(\rho, u, H)$ of $(\A, G,\alpha)$ is a strongly continuous unitary representation $u: G \rightarrow B(H)$ and a non-degenerate, completely contractive representation $\rho: \A \rightarrow B(H)$ satisfying the covariance relation
\[
u(g)\rho(a) = \rho(\alpha_g(a))u(g) 
\]
for all $g\in G$ and $a\in \A$.
Another covariant representation $(\sigma, v, K)$ of $(\A,G,\alpha)$ with $H\subset K$ is said to {\bfseries extend} $(\rho, u, H)$ if
\[
\sigma = \left[\begin{matrix}\rho & * \\ 0 & *\end{matrix}\right] \quad \textrm{and} \quad v = \left[\begin{matrix}u & 0 \\ 0 & *\end{matrix}\right],
\]
{\bfseries coextend} $(\rho, u, H)$ if
\[
\sigma = \left[\begin{matrix}\rho & 0 \\ * & *\end{matrix}\right] \quad \textrm{and} \quad v = \left[\begin{matrix}u & 0 \\ 0 & *\end{matrix}\right],
\]
and {\bfseries dilate} $(\rho, u, H)$ if
\[
\sigma = \left[\begin{matrix}* & * & * \\ 0& \rho & * \\ 0 & 0& *\end{matrix}\right] \quad \textrm{and} \quad v = \left[\begin{matrix}* & 0 & 0 \\ 0 & u & 0 \\ 0 & 0 & *\end{matrix}\right].
\]

In \cite[Theorem 4.6]{KatRamHN} it was proven that, for a C$^*$-correspondence $(X,\C)$, every covariant representation of $(\T_{(X,\C)}^+, G,\alpha)$  coextends to a covariant representation $(\sigma,v, K)$ where $\sigma$ is the integrated form of an isometric representation  of the correspondence. In light of Theorem \ref{thm:sDmaxisToeplitz} it turns out that this result is not specific to tensor algebras of C$^*$-correspondences but is true for dynamical systems of semi-Dirichlet operator algebras.

\begin{theorem}\label{thm:dilatestosD}
Let $(\A, G, \alpha)$ be a dynamical system where $\A$ is semi-Dirichlet. Every covariant representation $(\rho, u, H)$ coextends to a covariant representation $(\sigma, v, K)$ where $\sigma$ is semi-Dirichlet.
\end{theorem}
\begin{proof}
Much of this follows similarly to the proof of Proposition \ref{prop:shilov}. By \cite[Theorem 4.6]{DKDoc} $\rho$ has a unique cyclic extremal coextension $\sigma : \A \rightarrow B(K)$ which in turn has a maximal dilation (extension) $\pi$ on $B(L)$. Thus $\pi$ extends to a $*$-homomorphism of $C^*_e(\A)$. Thus, $\sigma$ is a Shilov representation of $\A$ and hence semi-Dirichlet by Proposition \ref{prop:shilov}.

It is an important point to recall that $\alpha$ extends uniquely to a continuous representation of $G$ on $\Aut(C^*_e(\A))$ \cite[Lemma 3.4]{KatRamMem}, which we will call $\alpha$ as well.

Since $\sigma$ is cyclic, $K = \overline{\sigma(\A)H}$, we can show for every $g\in G$ that $v(g)\in B(K)$ given by
\[
v(g)\sigma(a)h = \sigma(\alpha_g(a))u(g)h, 
\]
for all $a\in \A$ and $h\in H$, is a well-defined unitary.
To this end, for $a_1,a_2\in \A$ we know by the semi-Dirichlet property that there are $b_n,c_n\in \A$ such that 
\[
a_2^*a_1 = \lim_{n\rightarrow \infty} b_n + c_n^*.
\]
Hence, for $h_1,h_2\in H$
\begin{align*}
\Big\langle \sigma(\alpha_g(a_1))u(g)h_1, &\sigma(\alpha_g(a_2))u(g)h_2\Big\rangle
\\ & \quad = \ \ \Big\langle \pi(\alpha_g(a_1))u(g)h_1, \pi(\alpha_g(a_2))u(g)h_2\Big\rangle
\\ & \quad = \ \ \Big\langle \pi(\alpha_g(a_2^*a_1))u(g)h_1, u(g)h_2\Big\rangle
\\ & \quad = \ \ \lim_{n\rightarrow\infty} \Big\langle \pi(\alpha_g(b_n + c_n^*))u(g)h_1, u(g)h_2\Big\rangle
\\ & \quad = \ \ \lim_{n\rightarrow\infty} \Big\langle (\pi(\alpha_g(b_n)) + \pi(\alpha_g(c_n))^*)u(g)h_1, u(g)h_2\Big\rangle
\\ & \quad = \ \ \lim_{n\rightarrow \infty} \Big\langle (\rho(\alpha_g(b_n)) + \rho(\alpha_g(c_n))^*)u(g)h_1, u(g)h_2\Big\rangle
\\ & \quad = \ \ \lim_{n\rightarrow\infty} \Big\langle u(g)(\rho(b_n) + \rho(c_n)^*)h_1, u(g)h_2\Big\rangle
\\ & \quad = \ \ \lim_{n\rightarrow\infty} \Big\langle (\rho(b_n) + \rho(c_n)^*)h_1, h_2\Big\rangle
\\ & \quad \; \vdots
\\ & \quad = \ \ \Big\langle \sigma(a_1)h_1, \sigma(a_2)h_2\Big\rangle.
\end{align*}
Note that in the third last step we are using that since $u(g)\rho(a) = \rho(\alpha_g(a))u(g)$ then $u(g)\rho(a)^* = \rho(\alpha_g(a))^*u(g)$ by taking adjoints and moving the unitaries.
All of this implies that $v(g)$ is a well-defined isometry, in fact a unitary. 

Now for every $a\in \A$ and $h\in H$
\begin{align*}
P_H v(g)\sigma(a)h \ \ & = \ \ P_H \sigma(\alpha_g(a))u(g)h
\\ & = \ \ P_H \rho(\alpha_g(a))u(g)h
\\ & = \ \ P_H u(g) \rho(a)h.
\end{align*}
Thus, $P_H v(g) P_H = u(g)$ which means that there exists a unitary $\tilde u(g) \in B(K \ominus H)$ such that $v(g) = u(g) \oplus \tilde u(g)$, that is, $v(g)$ coextends $u(g)$.

As well, for $g_1,g_2\in G$ we have for every $a\in\A$ and $h\in H$
\begin{align*}
v(g_1g_2)\sigma(a)h \ \ & = \ \ \sigma(\alpha_{g_1g_2}(a))u(g_1g_2)h
\\ & = \ \ \sigma(\alpha_{g_1}(\alpha_{g_2}(a)))u(g_1)u(g_2)h
\\ & = \ \ v(g_1)v(g_2)\sigma(a)h.
\end{align*}
Hence, $v(g_1g_2) = v(g_1)v(g_2)$ and $v$ is a unitary representation of $G$. Strong continuity follows from the complete contractivity of $\sigma$, the point-norm continuity of $\alpha$ and the strong continuity of $u$.

Lastly, we should check the covariance relations. To this end, let $a,b\in\A$, $h\in H$ and $g\in G$
\begin{align*}
v(g)\sigma(a)\sigma(b)h \ \ & = \ \ u(g)\sigma(ab)h
\\ & = \ \ \sigma(\alpha_g(ab))u(g)h
\\ & = \ \ \sigma(\alpha_g(a))\sigma(\alpha_g(b))u(g)h
\\ & = \ \ \sigma(\alpha_g(a))v(g)\sigma(b)h.
\end{align*}
By the cyclicity of $K$ we then have $v(g)\sigma(a) = \sigma(\alpha_g(a))v(g)$.

Therefore, $(\sigma, v, K)$ is a covariant representation of $(\A, G, \alpha)$ coextending $(\rho, u, H)$ where $\sigma$ is semi-Dirichlet.

\end{proof}

We now turn to semi-Dirichlet property and crossed product operator algebras.
For a C$^*$-dynamical system $(\C, G, \alpha)$ recall that the continuous compactly-supported functions from $G$ into $\C$ is denoted $C_c(G,\C)$ and is a $*$-algebra under the convolution product and a compatible involution. The selfadjoint crossed products are completions of this $*$-algebra under the integrated forms of various covariant representations of the systems. There is a lot of theory here and the reader will have to consult \cite[Chapter 2]{Williams} for the necessary crossed product background.

\begin{definition}[\cite{KatRamMem}]
Suppose $(\A,G,\alpha)$ is a dynamical system and $(\C,\iota)$ is an $\alpha$-admissible C$^*$-cover. The  {\bfseries relative full crossed product} is defined to be
\[
\A \rtimes_{(\C,\iota),\alpha} G := \overline{C_c(G,\iota(\A))} \subseteq \C \rtimes_\alpha G,
\]
and the {\bfseries full crossed product} is
\[
\A \rtimes_\alpha G := \A \rtimes_{(\cmax(\A),\mu),\alpha} G.
\]
\end{definition}

In the same manner, one can define the relative reduced crossed product by taking the closure of $C_c(G,\iota(\A))$ in the reduced crossed product C*-algebra $\C \rtimes^r_\alpha G$. It is proven in \cite[Theorem 3.12]{KatRamMem} that all of these algebras are completely isometrically isomorphic and so we refer to this algebras as the {\bfseries reduced crossed product} and denote it $\A \rtimes^r_\alpha G$. As well, if $G$ is amenable then the full and reduced crossed products are completely isometrically isomorphic \cite[Theorem 3.14]{KatRamMem}.

\begin{proposition}
Let $(\A, G, \alpha)$ be a dynamical system and suppose $(\C_1,\iota_1)$ and $(\C_2, \iota_2)$ are $\alpha$-admissible C$^*$-covers. If $(\C_1,\iota_1) \preceq (\C_2, \iota_2)$ via a morphism $\varphi$, then there exists a completely contractive surjective homomorphism
\[
q_\varphi : \A \rtimes_{(\C_2, \iota_2),\alpha} G \rightarrow \A \rtimes_{(\C_1, \iota_1),\alpha} G
\]
such that $q_\varphi(f) = \varphi f \in C_c(G,\iota_1(\A))$ for all $f\in C_c(G,\iota_2(\A))$ (aka the identity map). Moreover, equivalence of C$^*$-covers implies that $q_\varphi$ is a  completely isometric isomorphism. 
\end{proposition}

\begin{proof}
This follows from similar reasoning as \cite[Theorem 2.4]{KatRamHN}.
By definition, the morphism $\varphi$ is a $*$-homomorphism $\varphi : \C_2 \rightarrow \C_1$ such that $\varphi\iota_2 = \iota_1$. As well, by the uniqueness of the extension of the group homomorphism to the C$^*$-covers we have $\varphi\alpha_g = \alpha_g\varphi$ (abusing notation by calling all the group homomorphisms $\alpha$). In other words, $\ker \varphi$ is an $\alpha$-invariant ideal (cf. \cite[Chapter 2]{Hamidi}). 

Now \cite[Proposition 3.19]{Williams} states that full crossed products of C$^*$-algebras preserve exact sequences by $\alpha$-invariant ideals. Hence,
\[
\varphi \rtimes \id : \C_2 \rtimes_\alpha G \rightarrow \C_1 \rtimes_\alpha G
\]
is a surjective $*$-homomorphism. In particular, $q_\varphi := \varphi\rtimes\id|_{\A \rtimes_{(\C_2,\iota_2), \alpha} G}$
is the required completely contractive surjective homomorphism.
\end{proof}

From the previous proposition we see that there is a complete lattice of relative full crossed products indexed by the complete sublattice of $\cstarlattice(\A)$ of $\alpha$-admissible C$^*$-covers. Of course, inequivalent C$^*$-covers may lead to completely isometrically isomorphic relative full crossed products, perhaps even collapsing to a single unique algebra as in the case when $G$ is amenable.

By \cite[Proposition 3.7-8]{KatRamMem} every non-degenerate completely contractive representation of $\A \rtimes_\alpha G$ arises precisely as the integrated form
\[
\rho \rtimes u : \A \rtimes_\alpha G \rightarrow B(H)
\]
of a covariant representation $(\rho, u, H)$ of $(A, G,\alpha)$. Thus, the full crossed product is characterized as the universal operator algebra over all such covariant representations \cite[Theorem 3.10]{KatRamMem}. 

In \cite[Theorem 5.8]{KatRamMem} the second and third authors proved that given a dynamical system $(\A, G,\alpha)$ where $\A$ is unital and semi-Dirichlet then the relative full and reduced crossed products, relative to the C*-envelope, are semi-Dirichlet as well. This result was what was needed to find the non-tensor algebra semi-Dirichlet examples mentioned in Section \ref{sec:semi-dirichlet}. It turns out that this result is true using any semi-Dirichlet admissible C*-cover and without the unital assumption.

First, we need a small discussion about unitization. Let $(\A, G, \alpha)$ be a dynamical system and let $(\C, j)$ be a C*-cover of $\A$ which is $\alpha$-admissible. If $\A$ is non-unital, then $\C$ is non-unital as well as a C*-algebra. Let $\C_1$ denote the unitization of $\C$. We have a natural inclusion $\C \rtimes_{\alpha} G \subseteq \C_1 \rtimes_{\alpha} G$ and under that inclusion, $\A \rtimes_{\C, j, \alpha} G$ is identified with the (closed) subalgebra of $ \C_1 \rtimes_{\alpha} G$ generated by $C_c(G, \A)$. By abusing notation, we call this copy of $\A \rtimes_{(\C, j), \alpha} G$ inside $ \C_1 \rtimes_{\alpha} G$ as $\A \rtimes_{(\C_1, j), \alpha} G$. Also the subalgebra of $ \C_1 \rtimes_{\alpha} G$ generated by $C_c(G, \bbC I )$ will be denoted as $\bbC I \rtimes_{\alpha} G$.

\begin{theorem}\label{thm:crossedissD}
Let $(\A, G, \alpha)$ be a dynamical system and let $(\C, j)$ be a C*-cover of $\A$ which is $\alpha$-admissible. If $(\C,j)$ is a semi-Dirichlet C$^*$-cover of $\A$, then 
\begin{itemize}
\item[\textup{(i)}] $\A \rtimes_{(\C, j), \alpha} G \subseteq \C \rtimes_{\alpha} G$ is semi-Dirichlet, and
\item[\textup{(ii)}]$\A\rtimes_{\alpha}^r G \subseteq \C \rtimes_{\alpha}^r G $ is semi-Dirichlet.
\end{itemize}
\end{theorem}

\begin{proof} If $(\A, G, \alpha)$ is a unital dynamical system then the same exact arguments as in the proof of \cite[Theorem 5.8]{KatRamMem} suffice to prove both (i) and (ii). So here we treat the case where $\A$ is non-unital.

(i) Let $(\pi, u, H)$ be a covariant representation of $(\C, G, \alpha)$ so that the representation $\pi \rtimes u$ is a faithful representation of $\C \rtimes_{\alpha} G$.

Let $H_1:=H \oplus H$ and consider
\[
\C_1:= \left\{ \left( \begin{smallmatrix} c+\lambda I & 0\\0&\lambda I \end{smallmatrix} \right) \mid c \in \C, \lambda \in \bbC\right\} \subseteq B(H_1)
\]
and 
\[
\A_1:= \left\{ \left( \begin{smallmatrix} a+\lambda I & 0\\0&\lambda I \end{smallmatrix} \right) \mid a \in \A, \lambda \in \bbC\right\} \subseteq \C_1.
\]
Note that $\A_1\subseteq C_1$ is semi-Dirichlet and so by the unital case 
\[
\A_1 \rtimes_{(\C_1, j), \alpha} G \subseteq \C_1 \rtimes_{\alpha} G
\]
is semi-Dirichlet. Let 
\[
\pi_1 : \C_1 \longrightarrow B(H_1) 
\]
be the inclusion map and consider the unitary representation 
\[
G \ni s \longmapsto v(s):=u(s)\oplus u(s) \in B(H_1).
\]
The pair $(\pi_1, v, H_1)$ forms a covariant representation of $(\C_1, G, \alpha)$. Let $\rho:=\pi_1\rtimes v$ be the integrated form representation of $\C \rtimes_{\alpha} G$.

Let $M, N\subseteq H_1$, with $M:=H\oplus 0$ and $N:=M^{\perp}$. For an elementary function $\lambda\otimes z \in C_c(G,\bbC I)$, with $\lambda \in \bbC$ and $z \in C_c(G)$, we have 
\begin{equation*}
\begin{split}
\rho(\lambda \otimes z)&= \lambda \int_{G} z(r)v(r)d\mu(r) \\
&= \lambda \int_{G} z(r)(u(r) \oplus u(r))d\mu(r) \\
&=   \Big( \lambda \int_{G} z(r)u(r) d\mu(r)  \Big)\oplus  \Big( \lambda \int_{G} z(r)u(r) d\mu(r)  \Big)
\\ & = (\rho(\lambda \otimes z) |_{N}) \oplus (\rho(\lambda \otimes z) |_{N}) .
\end{split}
\end{equation*}
Therefore the restriction of elements from $\rho(\bbC I \rtimes_{\alpha} G )$ onto $N$ is norm preserving.

On the other hand, $N$ is reducing for $\pi_1(\C_1)$ and in addition $\pi_1(\C) |_{N}=\{0\}$. This situation persists for elements of $\rho( \C\rtimes_{\alpha}G)$, when we view $ \C\rtimes_{\alpha}G$ as a subset of $ \C_1\rtimes_{\alpha}G$. Indeed, if $P$ is the orthogonal projection to $N$, then for any elementary function $c \otimes z \in C_c(\C,  G)$ with $c \in \C$ and $z \in C_c(G)$, we have
\begin{equation*}
\begin{split}
\rho(c \otimes z)P&=  \int_{G} \pi_1(\alpha^{-1}_r(c) )z(r)v(r) Pd\mu(r) \\
&=  \int_{G}\pi_1(\alpha^{-1}_r(c) ) P z(r)v(r) d\mu(r) \\
&=0 = P\rho(c \otimes z).
\end{split}
\end{equation*}
In addition, similar arguments show that the restriction of $\rho$ on  $\C\rtimes_{\alpha}G$ is a $*$-isomorphism.

For the the proof, note that 
\[
\A_1\rtimes_{(\C_1, j), \alpha}G = \overline{\spn}\big( \A \rtimes_{(\C, j), \alpha}G + \bbC I\rtimes_{\alpha} G\big)
\]
and so 
\begin{equation} \label{eq:Sub}
\S\big(\A_1\rtimes_{(\C_1, j), \alpha}G  \big) \subseteq \overline{\spn}\big( \S\big(\A \rtimes_{(\C, j), \alpha}G  \big)+ \bbC I\rtimes_{\alpha} G\big),
\end{equation}
where 
\[
\S(\B) = \overline{\B + \B^*}
\]
is a shorthand for the operator system generated by $\B$.
Let $x \in (\A \rtimes_{(\C, j), \alpha} G)^*(\A \rtimes_{(\C, j), \alpha} G)$. Since $\A_1 \rtimes_{(\C_1, j), \alpha} G$ is semi-Dirichlet, we obtain from (\ref{eq:Sub}) sequences $\{y_n\}_{n=1}^{\infty}$ and  $\{w_n\}_{n=1}^{\infty}$ in $ \S\big(\A \rtimes_{(\C, j), \alpha}G\big)$ and $\bbC I\rtimes_{\alpha} G$ respectively, so that 
\[
x=\lim_n (y_n +w_n).
\]
By compressing to $N$, which annihilates $\rho( \C\rtimes_{\alpha}G)$, the above equation implies that \\ $\lim_n \rho( w_n) |_{N}=0$. As we have seen however, the compression on $N$ is norm preserving for elements of $\rho\big(\bbC I\rtimes_{\alpha}G\big)$ and so $\lim_n \rho(w_n)=0$. Therefore $\rho(x)=\lim_n \rho(y_n)  \in \rho\big( \S\big(\A \rtimes_{(\C, j), \alpha}G\big)\big)\subseteq  \S\big(\rho(\A \rtimes_{(\C, j), \alpha}G)\big)$ and so $\rho(\A \rtimes_{(\C, j), \alpha}G)$ is semi-Dirichlet. However, $\rho$ is a $*$-isomorphism on $\C\rtimes_{\alpha} G \subseteq \C_1\rtimes_{\alpha} G$ and so $ \A \rtimes_{(\C, j), \alpha}G$ is semi-Dirichlet.

(ii) Assume that $\A$ acts on a Hilbert space $H_0$ so that $\A$ is semi-Dirichlet as a subalgebra of $B(H_0)$ and $\ca(\A) \subseteq B(H_0)$ is an admissible C*-cover of $\A$ with respect to the $\alpha$-action of $G$. Let $H:= H_0\oplus \bbC$ and consider 
\[
\A_1:= \left\{ \left( \begin{smallmatrix} a+\lambda I & 0\\0&\lambda \end{smallmatrix} \right) \mid a \in \A, \lambda \in \bbC\right\}
\]
be the unitization of $\A$. Since $\A$ is semi-Dirichlet as a subalgebra of $B(H_0)$, it is easy to see that $\A_1$ is semi-Dirichlet as a subalgebra of $B(H)$. Furthermore, $\ca(\A_1)$ is an admissible C*-cover for the unitized $\alpha$-action of $G$. Therefore, $\A_1\rtimes_{\alpha}^r G$ is semi-Dirichlet. Lets look at the reduced crossed product $\A_1\rtimes_{\alpha}^r G$ more carefully.

Let
\[
\pi : \A_1 \longrightarrow B(H) 
\]
be the inclusion map and define
\begin{equation*}
\begin{split}
\tilde{\pi}(c)h(r) &:=\pi(\alpha_r^{-1}(c))h(r)\\
u(s) h(r) &:= h(s^{-1}r)
\end{split}
\end{equation*}
for $c \in C^\ast(\A_1)$, $s,r \in G$ and $h \in H\otimes L^2(G)$. Then, $(\tilde{\pi}, u, H)$ forms a covariant representation of $(\ca(\A_1), G,\alpha)$ and by definition  the integrated form representation $\rho:= \tilde{\pi} \rtimes u$ maps $\ca(\A_1) \rtimes_{\alpha} G$ onto $\ca(\A_1) \rtimes_{\alpha}^r G$. Consider the decomposition 
\begin{equation} \label{eq;split}
\tilde{H}:=H\otimes  L^2(G) = (H_0\otimes  L^2(G)) \oplus (\bbC \otimes  L^2(G)) 
\end{equation}
Now $u(r) = I \otimes \ell(r)$, $r \in G$, where $\ell$ is the left regular representation of $G$, and the restriction of $u$ to $N=\bbC \otimes L^2(G)$ is unitarily equivalent to $\ell$. Because $\ell$ is faithful on the reduced C*-algebra $C^\ast_r(G)\cong\bbC I\rtimes_{\alpha}^r G$, the restriction of elements from $\bbC I \rtimes_{\alpha}^rG \subseteq \A_1 \rtimes_{\alpha}^rG$ to $N$ is norm preserving.

On the other hand, $N$ is also reducing for $\tilde{\pi}(\A_1)$ and in addition $\tilde{\pi}(\A) |_{N}=\{0\}$. This equality persists for elements of $\A \rtimes_{\alpha}^rG \subseteq \A_1 \rtimes_{\alpha}^rG$. Indeed, if $P$ is the orthogonal projection on $N$, then for any elementary function $a \otimes z \in C_c(\A, G)$, with $a \in \A$ and $z \in C_c(G)$, we have
\begin{equation*}
\begin{split}
\rho(a \otimes z)P&=  \int_{G} \alpha^{-1}_r(a) z(r)u(r) Pd\mu(r) \\
&=  \int_{G} \alpha^{-1}_r(a) P z(r)u(r) d\mu(r) =0
\end{split}
\end{equation*}

To start with the proof, note that 
\[
\A_1\rtimes_{\alpha}^rG = \overline{\spn}\big( \A \rtimes_{\alpha}^rG + \bbC I\rtimes_{\alpha}^rG\big)
\]
and so 
\begin{equation} \label{eq;S}
\S\big(\A_1\rtimes_{\alpha}^rG \big) = \overline{\spn}\big( \S\big(\A \rtimes_{\alpha}^rG \big)+ \bbC I\rtimes_{\alpha}^rG\big).
\end{equation}
Let $x \in (\A \rtimes_{\alpha}^rG)^*(\A \rtimes_{\alpha}^rG)$. Since $\A_1\rtimes_{\alpha}^rG$ is semi-Dirichlet, we obtain from (\ref{eq;S}) sequences $\{y_n\}_{n=1}^{\infty}$ and  $\{w_n\}_{n=1}^{\infty}$ in $\S\big(\A \rtimes_{\alpha}^rG \big)$ and $\bbC I\rtimes_{\alpha}^rG$ respectively, so that 
\[
x=\lim_n (y_n +w_n).
\]
By compressing to $N$, which annihilates $\A\rtimes_{\alpha}^rG$, the above equation implies that \\ $\lim_n w_n |_{N}=0$. As we have seen however, the compression on $N$ is norm preserving for elements of $\bbC I\rtimes_{\alpha}^rG$ and so $\lim_n w_n=0$. Therefore $x=\lim_n y_n  \in \S\big(\A \rtimes_{\alpha}^rG \big)$ and so $\A \rtimes_{\alpha}^rG$ is semi-Dirichlet.
\end{proof}





In \cite[Theorem 5.5]{KatRamMem} it was proven that if $\A$ is Dirichlet then so is $\A \rtimes_\alpha G$. We are now in a position to prove this result in the semi-Dirichlet case.

\begin{theorem}\label{thm:sDcrossedproduct}
Suppose $(\A, G, \alpha)$ is a dynamical system. If $\A$ is semi-Dirichlet, then $\A \rtimes_\alpha G$ is completely isometrically isomorphic to $\A \rtimes_{(\sDmax(\A),\nu),\alpha} G$ via the canonical map. Moreover, $\A \rtimes_\alpha G$ is semi-Dirichlet.
\end{theorem}

\begin{proof}
Let $\pi : \A \rtimes_\alpha G \rightarrow B(H)$ be a completely isometric non-degenerate representation.
By \cite[Proposition 3.8]{KatRamMem}, $\pi$ is the integrated form of a non-degenerate covariant representation $(\rho, u, H)$, $\pi = \rho\rtimes u$. Necessarily, $\rho$ is completely isometric, too.
Theorem \ref{thm:dilatestosD} gives that $(\rho, u ,H)$ coextends to a covariant representation $(\sigma, v, K)$ where $\sigma$ is semi-Dirichlet. Moreover, since $P_H\sigma|_H = \rho$ then $\sigma$ is also completely isometric. Thus, by the universal property of $\sDmax(\A)$ there exists a $*$-homomorphism $\tilde\sigma : \sDmax(\A) \rightarrow B(K)$ such that $\tilde\sigma(\nu(a)) = \sigma(a)$ for all $a\in \A$.

Now $\alpha$ extends uniquely to $\sDmax(\A)$ by Proposition \ref{prop:sDmaxAdmissible} and so
\begin{align*}
v(g)\tilde\sigma(\nu(a))v(g)^* & = v(g)\sigma(a)v(g)^* 
\\ & = \sigma(\alpha_g(a) )
\\ & = \tilde\sigma(\nu(\alpha_g(a)))
\\ & = \tilde\sigma(\alpha_g(\nu(a)))
\end{align*}
for all $a\in \A$.
Hence, $v(g)\tilde\sigma(c) = \tilde\sigma(\alpha_g(c))v(g)$ for all $c\in \sDmax(\A)$ by the uniqueness of the universal property of $\sDmax(\A)$. Thus, $(\tilde\sigma, v, K)$ is a covariant representation of $(\sDmax(\A), G, \alpha)$.

To finish off the argument remember that $H\subset K$ is coinvariant for $\tilde\sigma$ and $v$. Then compression to $H$ is a completely contractive homomorphism of $\A\rtimes_{(\sDmax(\A),\nu),\alpha} G$
onto $\A \rtimes_\alpha G$. Therefore, by the universality of the full crossed product these two operator algebras are completely isometrically isomorphic.

Lastly, by the previous theorem we have that $\A \rtimes_\alpha G$ is a semi-Dirichlet operator algebra.
\end{proof}

We end this section with a nice application of Takai duality.

\begin{theorem}\label{thm:AbeliansDcrossed}
Suppose $(\A, G, \alpha)$ is a unital dynamical system with $G$ a locally compact abelian group. Then $\A$ is semi-Dirichlet if and only if $\A \rtimes_\alpha G$ is semi-Dirichlet.
\end{theorem}

\begin{proof}
The previous theorem gives the forward direction. Assume then, that $\A \rtimes_\alpha G$ is semi-Dirichlet. 
Thus, the full crossed product of the dual dynamical system $(\A\rtimes_\alpha G, \hat G, \hat \alpha)$ \cite[Chapter 4]{KatRamMem} is semi-Dirichlet by Theorem \ref{thm:sDcrossedproduct}. The Takai duality of crossed products of operator algebras \cite[Theorem 4.4]{KatRamMem} then gives
\[
(\A \rtimes_\alpha G) \rtimes_{\hat\alpha} \hat G \simeq \A \otimes \K(L^2(G)),
\]
where the latter algebra is a subalgebra of $C^*_e(\A) \otimes \K(L^2(G))$. In particular, since $\K(L^2(G))$ is simple and nuclear then everything with the tensor product behaves nicely and from \cite[Proposition 4.4]{HumRam},
\[
C^*_e(\A) \otimes \K(L^2(G)) \simeq C^*_e(\A\otimes \K(L^2(G))).
\]
Thus, the semi-Dirichlet property gives
\[
\left(\A\otimes \K(L^2(G))\right)^*\left(\A\otimes \K(L^2(G))\right) \ \subseteq \ \overline{\A\otimes \K(L^2(G)) + \left(\A\otimes \K(L^2(G))\right)^*}
\]
in $C^*_e(\A) \otimes \K(L^2(G))$ which can be treated as a subalgebra of $C^*_e(\A)_1 \otimes \K(L^2(G))$ since the nuclearity of the second algebra makes everything work out nicely.
So for $a,b\in \A$ and a rank-one projection $p\in \K(L^2(G))$ there exist $c_n,d_n\in \A \otimes \K(L^2(G))$ such that
\[
(a\otimes p)^*(b\otimes p) = \lim_{n\rightarrow \infty} c_n + d_n^*.
\]
But then
\begin{align*}
a^*b \otimes p  \ & = \ (1\otimes p)(a\otimes p)^*(b\otimes p)(1\otimes p)
\\ & = \ (1\otimes p)\left(\lim_{n\rightarrow \infty} c_n + d_n^*\right)(1\otimes p)
\\ & = \ \lim_{n\rightarrow \infty} (1\otimes p)c_n(1\otimes p) + (1\otimes p)d_n^*(1\otimes p).
\end{align*}
This implies that 
\[
(\A \otimes p)^*(\A \otimes p) \subseteq \overline{(\A \otimes p) + (\A \otimes p)^*}.
\]
Therefore, $\A$ is semi-Dirichlet.
\end{proof}

\section*{Acknowledgements}
Adam Humeniuk was partially supported as a Postdoctoral Fellow at MacEwan University, by the third author, Nicolae Strungaru, and Adi Tcaciuc.
Elias Katsoulis was partially supported by the NSF grant DMS-2054781.
Christopher Ramsey was supported by the NSERC Discovery grant 2019-05430. The authors would like to thank the reviewer for their careful reading and improvements.

\end{document}